\documentclass[11pt]{article}
\usepackage{pstricks}
\usepackage{bbm}
\usepackage{amsmath,amssymb,amsthm}
\usepackage{amsbsy}
\usepackage{epsfig}
\usepackage{cite}

\def\x{x}

\newtheorem{theorem}{Theorem}[section]

\newtheorem{corollary}[theorem]{Corollary}
\theoremstyle{definition}

\newtheorem{remark}[theorem]{Remark}

\numberwithin{equation}{section}
\numberwithin{theorem}{section}
\newcommand{\comment}[1]{}

\begin{document}{}

\title{A Class of Moving Boundary Problems with a Source Term. \\ Application of a Reciprocal Transformation}

\author{Adriana C. Briozzo$^{1,2}$, Colin Rogers$^{3}$, Domingo A. Tarzia$^{1,2}$\\[5pt]
(1) Departamento. de Matem\'atica, FCE Universidad Austral, \\Paraguay 1950, Rosario, Argentina\\
(2) CONICET, Argentina  \\ \texttt{abriozzo@austral.edu.ar, dtarzia@austral.edu.ar}\\(2) School of Mathematics and Statistics, The University of New South Wales, \\Sydney, NSW2052, Australia\\
\texttt{c.rogers@unsw.edu.au}}
\maketitle

\begin{abstract} 
We consider a new Stefan-type problem for the classical heat equation with a latent heat and phase-change temperature depending of the variable time. We prove the equivalence of this Stefan problem with a class of boundary value problems for the nonlinear canonical evolution equation involving a source term with two free boundaries. This equivalence is obtained by applying a reduction to a Burgers equation and a reciprocal-type transformations. Moreover, for a particular case, we obtain a unique explicit solution for the two different problems.
\end{abstract}

\section{Introduction}

In \cite{pbcr93}, a systematic search was undertaken via Lie-B\"acklund transformations for classes of nonlinear evolution equations which are reducible to a canonical form as originally set down in \cite{nfjs89} which incorporates a source term. This nonlinear equation was shown in \cite{nfjs89} to admit reduction to Burgers equation and an explicit pulse solution with compact support together with an analytic description of interaction between pulses were thereby derived. In \cite{pbcr93}, an extension of the Fokas-Yortsos equation with convective term of \cite{afyy82,crmsdc83} was derived via the Lie-B\"acklund analysis and which incorporates a novel reaction term. This nonlinear evolution equation was shown to be relevant to the modelling of unsaturated flow in a soil with a volumetric extraction mechanism. In \cite{pbcr93}, a reciprocal transformation was used to solve a nonlinear boundary-value problem incorporating a source term descriptive of transient flow in a finite layer of soil subject to a constant flux boundary condition to compensate for water extraction.

Here by contrast, our concern is with a class of moving boundary problems pertinent to soil mechanics. An inverse procedure is adopted whereby a class of boundary value problems for the canonical nonlinear evolution equation of \cite{nfjs89} involving a source term is shown to be amenable to analytic solution by application of a reciprocal link to a Stefan-type problem for the classical heat equation.

It is recalled that moving boundary problems of Stefan-type have their origin in the analysis of the melting of solids and the freezing of liquids (see e.g. \cite{lr71,af82,cejo82,jc84,dt00,dt11} and literature cited therein). The standard Stefan problems concern moving boundary problems for the classical linear heat equation where the heat balance requirement on the moving boundary separating the phases leads to a nonlinear boundary condition on the temperature.

In \cite{abdt20a,abdt20b,fcsl94,masl00,masl2000}, a novel integral representation version of the Hopf-Cole transformation was used to treat certain classes of Stefan-type problems for Burgers equation. Reciprocal-type transformations on the other hand have been previously applied to solve a wide range of nonlinear moving boundary problems such as arise in nonlinear heat conduction, the analysis of the melting of metals, sedimentation and other physical contexts \cite{cr85,crtr85,cr86,crpb88,crbg88,crpb92,afcrws05,cr15,cr2015,cr16,cr17,cr19,cr2019}. It is remarked that the results in \cite{cr2015,cr16,cr17} obtained via application of reciprocal transformations concern moving boundary problems for certain solitonic equations, namely, the Dym, potential mkdV as well as the extended Dym equation derived via geometric considerations in \cite{wscr99}. In the present work, the integral representation of \cite{fcsl94} is allied with a reciprocal-type transformation to reduce to canonical form a class of moving boundary problems relevant to the soil mechanics context as described in \cite{pbcr93}.
In the following Section 2 we give a connection between a Stefan problem with variable latent heat term \cite{jbdt18a,jbdt18b,jbdt18c,jbdt18d,jbdt21,vvjs04,yzwbml13,yzywwb14,yzlx15} and time dependent temperature on the free boundary, and a moving boundary problem for the Burgers equation.
In Section 3 we use the reciprocal-type transformation to proof the equivalence between the moving boundary problem governed by Burgers equation and a free boundary problem governed by the nonlinear evolution equation with source term and two free boundaries. Then we give a parametric expression to solution of this problem through the solution to the Stefan problem.
Finally in Section 4 we solve the canonical Stefan problem for a particular latent heat and phase change temperature variable in time and we obtain an explicit solution of the similarity type. At last, we express the parametric solution to the nonlinear evolution problem with source and two free boundaries.
 
\section{A Canonical Connection}

Here, a connection is established between a classical Stefan-type problem but with variable latent heat term and a class of moving boundary problems for the nonlinear transport equation incorporating a source term of \cite{nfjs89}. The canonical Stefan problem to be considered here adopts the form
\begin{equation} \label{b1}
\begin{array}{l} T_t = T_{yy} \ , \quad 0 < y < S(t) \ , \quad t > 0
\\[2mm]-T_y(S(t),t) = L(t) \dot{S}(t) \ , \quad t > 0 \\[2mm]
T(S(t),t)= T_m(t) \ , \quad t > 0 \\[2mm]T_y(0,t)=-q \ , \quad q>0 \ , \quad t > 0 \\[2mm]
S(0) = 0 \ . \end{array} \end{equation}

On introduction of the integral representation of \cite{fcsl94}, namely
\begin{equation} \label{b2}
x^*=-1/\delta\left[\ln|C(t)-\int^y_{S(t)}T(y',t)dy'|\right]_y \ , \end{equation}
with $C(t)>0, C(0)=0=$, $\delta>0$ and 
\begin{equation} \label{b3}
T=\delta C(t)x^*(y,t)\exp\left[-\delta\int^y_{S(t)}x^*(\sigma,t)dt\right] \end{equation}
then it may be shown that (see, \textit{in extenso} \cite{cr15}) that the relations \eqref{b2} and \eqref{b3} link the classical heat equation $T_t=T_{yy}$ to the Burgers equation
\begin{equation} \label{b4}
x^*_t=x^*_{yy}-2\delta x^*x^*_y \end{equation}
if it is required that
\begin{equation} \label{b5}
\dot{C}+T_y|_{y=S(t)}+\dot{S}T|_{y=S(t)}=0 \ . \end{equation}
Thus, in view of the moving boundary conditions in \eqref{b1} it is seen that $C(t)$ is here determined via the relation
\begin{equation} \label{b6}
\dot{C}(t)=[\ L(t)-T_m(t)\ ] \dot{S}(t) \ . \end{equation}

The boundary requirements on $y=S(t)$ in the Stefan problem \eqref{b1}, in turn, become for the associated Burgers equation \eqref{b4}
\begin{equation} \label{b7}
(x^*_y-\delta x^{*2})|_{y=S(t)}=-L(t)\dot{S}(t)/\delta C(t) \end{equation}
and
\begin{equation} \label{b8}
x^*|_{y=S(t)}=T_m(t)/\delta C(t) \end{equation}
respectively. The condition on fixed face $y=0$ becomes 
\begin{equation} \label{b9}
x^*_y(0,t)-\delta x^{*2}(0,t)=\frac{-q \exp\left(\delta\int^0_{S(t)}x^*(\sigma,t)d\sigma\right)}{\delta C(t)}\end{equation} and the initial condition $S(0)=0$.

\section{Application of a Reciprocal Transformation}
Here we will consider the  nonlinear evolution equation with source term  \cite{nfjs89}, namely
\begin{equation}\label{C0}\frac{\partial\Psi}{\partial t}=\frac{\partial}{\partial x^*}\left(\frac{\Psi_{x^*}}{\Psi^2}\right)+2\delta  
\end{equation}

We obtain the following result:
\begin{theorem}
If $(x^*(y,t),S(t))$ is solution of the problem given by \eqref{b4},\eqref{b6}-\eqref{b9} then the function  $$\Psi(x^*,t)=\frac{1}{x^*_y(y,t)}$$ satisfies the following problem governed by the nonlinear evolution equation with source term, given by 

\begin{equation}\label{C}\frac{\partial\Psi}{\partial t}=\frac{\partial}{\partial x^*}\left(\frac{\Psi_{x^{*}}}{\Psi^2}\right)+2\delta  \ , \quad X_0^*(t)< x^* < X_1^*(t) \ , \quad t > 0
\end{equation} with the conditions

\begin{equation} \label{c4}
\begin{array}{l}  (i)\quad   \frac{1}{\Psi(X_1^*(t),t)}-\delta{X_1^{*}}^{2}(t) = -\frac{L(t)}{\delta C(t)}\left(\Psi(X_1^*(t),t) \dot{X_1^*}(t)+\frac{\Psi_{x^*}(X_1^*(t),t)}{\Psi^2(X_1^*(t),t)}+2\delta X_1^*(t)\right)\ , \quad t > 0 \\[3mm]
(ii)  \quad\frac{1}{\Psi(X_0^*(t),t)}-\delta{X_0^*}^{2}(t) =\frac{-q \exp\left(\int^{t}_{0}H(\tau)d\tau\right)}{\delta C(t)} \ , \quad q>0 \ , \quad t > 0 \\[3mm]
(iii) \quad \dot{C}(t)=[\ L(t)-T_m(t)\ ] \left[\Psi(X_1^*(t),t) \dot{X_1^*}(t)+\frac{\Psi_{x^*}(X_1^*(t),t)}{\Psi^2(X_1^*(t),t)}+2\delta X_1^*(t)\right] \ , \end{array} \end{equation}

where the moving boundaries $X_0^*$ and $X_1^*$ are given by
\begin{equation} \label{c5}
X_0^*(t)=-\int^t_{0} \left(\frac{\Psi_{x^*}(X_0^*(\tau),\tau)}{\Psi^3(X_0^*(\tau),\tau)}+2\delta \frac{X_0^*(\tau)}{\Psi(X_0^*(\tau),\tau)}\right)d\tau \quad,\quad X_1^*(t)= T_m(t)/\delta C(t) \ , \quad t > 0\end{equation}
and
\begin{equation}\label{ache}
H(t)=-\frac{T_m^3(t)}{L(t)C^{2}(t)}+\delta\frac{T_m(t)}{L(t)}\frac{1}{\Psi(X_1^*(t),t)}-\frac{\delta}{\Psi(X_1^*(t),t)}+\frac{T_m^2(t)}{C^{2}(t)}+\frac{\delta}{\Psi(X_0^*(t),t)}-\delta^2 {X_0^*}^{2}(t)
\end{equation}

\end{theorem} 

\begin{proof}
Let\begin{equation} \label{c1}
\rho(y,t)=x^*_y(y,t)\end{equation}
and we define \begin{equation} \label{c2}
\Psi(x^*,t)=\frac{1}{\rho(y,t)}.\end{equation}
From \eqref{b4} we known that \begin{equation} \label{c3}
dx^*=\rho dy+(\rho_y-2\delta\rho x^*)dt,\end{equation}
and from \eqref{c2}, we have \begin{equation}\rho_y=-\frac{\Psi_{x^*}}{\Psi^2}x^*_y=-\frac{\Psi_{x^*}}{\Psi^3}
\end{equation}
then \begin{equation} \label{c6}
dx^*=\Psi^{-1}dy-(\Psi_{x^*}/\Psi^3+2\delta x^*/\Psi)dt \end{equation}
thus the reciprocal transformation is given by
\begin{equation} \label{c7}
dy=\Psi dx^*+(\Psi_{x^*}/\Psi^2+2\delta x^*)dt . \end{equation}
and we have 
\begin{equation} 
\dfrac{\partial y}{\partial x^*}=\Psi \ , \quad \dfrac{\partial y}{\partial t}=\dfrac{\Psi_{x^*}}{\Psi^2}+2\delta x^*\end{equation}

Therefore \eqref{C} follows. From \eqref{c1} we can write
\begin{equation} 
x^*(y,t)=\int^y_0 \rho(\sigma,t)d\sigma+M(t)\end{equation} then
\begin{equation} 
x_t(y,t)=\int^y_0 \rho_t(\sigma,t)d\sigma+M'(t)=\int^y_0 \left(x^*_{\sigma\sigma}(\sigma,t)-2\delta x^*(\sigma,t)x^*_\sigma(\sigma,t)\right)_\sigma d\sigma+M'(t)\end{equation}which implies that 
\begin{equation}
M'(t)=x^*_{yy}(0,t)-2\delta x^*(0,t)x^*_y(0,t)
\end{equation} therefore
\begin{equation} \label{equis}
x^*(y,t)=\int^y_0 \rho(\sigma,t)d\sigma+\int^t_0 \left[x^*_{yy}(0,\tau)-2\delta x^*(0,\tau)x^*_y(0,\tau)\right]d\tau\end{equation}
For $y=0$, we obtain 
\begin{equation}
x^*(0,t)=\int^t_0 \left[x^*_{yy}(0,\tau)-2\delta x^*(0,\tau)x^*_y(0,\tau)\right]d\tau,
\end{equation}
if we notice $X^*_0(t):=x^*(0,t)$ and taking into account \eqref{c1} -\eqref{c3} we have
\begin{equation} \label{c55}
X_0^*(t)=-\int^t_{0} \left(\frac{\Psi_{x^*}(X_0^*(\tau),\tau)}{\Psi^3(X_0^*(\tau),\tau)}+2\delta \frac{X_0^*(\tau)}{\Psi(X_0^*(\tau),\tau)}\right)d\tau .\end{equation} From \eqref{b8} we obtain \begin{equation}X^*_1(t):=x^*(S(t),t)=T_m(t)/\delta C(t)
\end{equation}
Then the nonlinear evolution equation with source term  \cite{nfjs89}, 
given by \eqref{C} is obtained in the domain $X^*_0(t)<x^*<X^*_1(t)$, $t>0$.

To prove \eqref{c4}(i), we consider condition \eqref{b7} which is equivalent to
\begin{equation} \label{condfronteralibre}
\dfrac{1}{\Psi(X^*_1(t),t)}-\delta {X^*}^{2}_1(t)=-L(t)\dot{S}(t)/\delta C(t). \end{equation} 
Now we must write $\dot{S}(t)$ as a function of the new variables and free boundaries. From \eqref{equis} we obtain
\begin{equation} 
X^*_1(t)=x^*(S(t),t)=\int_0^{S(t)}\rho(\sigma,t)d\sigma+\int^t_0 \left[x^*_{yy}(0,\tau)-2\delta x^*(0,\tau)x^*_y(0,\tau)\right]d\tau\end{equation}
then
\begin{equation} 
\dot{X}^*_1(t)=\dfrac{\dot{S}(t)}{\Psi(X^*_1(t),t)}+x^*_{yy}(S(t),t)-2\delta x^*(S(t),t)x^*_y(S(t),t)
\end{equation}and taking into account
\[x^*_{yy}(S(t),t)-2\delta x^*(S(t),t)x^*_y(S(t),t)=-\Psi_{x^*}(X^*_1(t),t)/\Psi^3(X^*_1(t),t)-2\delta X^*_1(t)/\Psi(X^*_1(t),t)
\] we obtain
\begin{equation}\label{esepunto}
\dot{S}(t)=\Psi(X^*_1(t),t))\dot{X}^*_1(t)+\Psi_{x^*}(X^*_1(t),t)/\Psi^2(X^*_1(t),t)+2\delta X^*_1(t)
\end{equation}
We replace $\dot{S}(t)$ in \eqref{condfronteralibre} to obtain \eqref{c4} (i).

 Next, we will deduce \eqref{c4}(ii). From \eqref{b9} we have 
 \begin{equation} 
\frac{1}{\Psi(X^*_0(t),t)}-\delta {X^*_0}^2(t)=\frac{-q \exp\left(\delta\int^0_{S(t)}x^*(\sigma,t)dt\right)}{\delta C(t)}.\end{equation} 

We define 
 $$R(t)=\exp\left(-\delta\int^{S(t)}_0 x^*(\sigma,t)dt\right)$$ then
 \[ln(R(t))=-\delta\int^{S(t)}_0 x^*(\sigma,t)dt\] and
 \[\frac{R'(t)}{R(t)}=-\delta x^*(S(t),t)\dot{S}(t)-\delta\int^{S(t)}_0 x^*_t(\sigma,t)dt=-\delta x^*(S(t),t)\dot{S}(t)-\delta\int^{S(t)}_0
\left(x^*_{\sigma\sigma}(\sigma,t)-2\delta x^*(\sigma,t)\right)d\sigma\]
  \[=-\delta x^*(S(t),t)\dot{S}(t)-\delta x^*_{y}(S(t),t)+\delta^{2} {x^*}^{2}(S(t),t)+\delta x^*_{y}(0,t)-\delta^{2} {x^*}^{2}(0,t)\]
  \[=-\delta X^*_1(t)\dot{S}(t)-\delta x^*_{y}(S(t),t)+\delta^{2} {X^*_1}^{2}(t)+\delta x^*_{y}(0,t)-\delta^{2} {X^*_0}^{2}(t)\]
 \[=-\delta^2 \frac{X^*_1(t)C(t)}{L(t)}\left[\delta {X^*_1}^{2}(t)-\frac{1}{\Psi(X^*_1(t),t)} \right]- \frac{\delta}{\Psi(X^*_1(t),t)}+\delta^{2} {X^*_1}^{2}(t)+ \frac{\delta}{\Psi(X^*_0(t),t)}-\delta^{2} {X^*_0}^{2}(t).\]

Using \eqref{c5} we obtain \[\frac{R'(t)}{R(t)}=H(t)\]with $H$, given by \eqref{ache}, and integrating it results
\[ln(R(t))-ln(R(0))=\int^{t}_0 H(\tau)d\tau\] or equivalently
\[R(t)=exp\left(\int^{t}_0 H(\tau)d\tau\right)\] because $R(0)=1$.
Therefore \eqref{b9} is equivalent to \eqref{c4}(ii).

The condition \eqref{c4}(iii) yields immediately from \eqref{b6} when $\dot{S}(t)$ is replaced by the expression \eqref{esepunto}.
\end{proof}

Reciprocally through the reciprocal transformation given by
\begin{equation} \label{trans}
dy=\Psi dx^*+(\Psi_{x^*}/\Psi^2+2\delta x^*)dt.
 \end{equation} 
we will prove that if $\Psi(x^*,t)$ satisfies \eqref{C}-\eqref{ache} then the pair $(x^*(y,t),S(t))$ is solution to the problem \eqref{b4},\eqref{b6} -\eqref{b9} where
\begin{equation}\label{ese}
S(t)=\int_{X^*_0(t)}^{X^*_1(t)}\Psi(\sigma,t)d\sigma
\end{equation}

\begin{theorem}
If $\Psi(x^*,t),X^*_0(t),X^*_1(t) $ satisfy \eqref{C}-\eqref{ache} then the pair $(x^*=x^*(y,t),S(t))$ with both components are defined by \eqref{trans} and  \eqref{ese} respectively is solution to the problem \eqref{b4},\eqref{b6} -\eqref{b9}.
\end{theorem}

\begin{proof}
From \eqref{trans} we have 
\begin{equation} 
dx^*=\frac{1}{\Psi}dy-(\Psi_{x^*}/\Psi^3+ 2\delta x^*/\Psi)dt.\  \end{equation}
Taking $\rho(y,t)=\frac{1}{\Psi(x^*,t)}$ we have
 
\begin{equation} 
dx^*=\rho dy+(\rho_y-2\delta x^*\rho)dt  \end{equation}
which implies 
\begin{equation} 
x^*_y=\rho \ , \quad x^*_{t}=\rho_y-2\delta\rho x^* \end{equation}
whence
\begin{equation*}
x^*_{t}=x^*_{yy}-2\delta x^*x^*_y \ . \end{equation*}
Moreover, from \eqref{trans} we have
\begin{equation}
y=\int_{X^*_0(t)}^{x^*}\Psi(\sigma,t)d\sigma+N(t)
\end{equation}then 
\begin{equation}
\frac{\partial y}{\partial t}=-\Psi(X^*_0(t),t)\dot{X}^*_0(t)+\int_{X^*_0(t)}^{x^*}\Psi_t(\sigma,t)d\sigma+N'(t)
\end{equation} and taking into account \eqref{C}, \eqref{c5} and the fact that 
\begin{equation}
\frac{\partial y}{\partial t}=\Psi_{x^*}/\Psi^2+2\delta x^* \end{equation} 
we obtain $N'(t)=0$ which implies that $N(t)$ is a constant which we take null. Therefore for $x^*=X^*_0(t)$ results $y=0$ and for $x^*=X^*_1(t)$ we obtain $y=S(t)$ defined by \eqref{ese}.

Derivating in \eqref{ese} we have
\begin{equation}\label{esepunto2}
\dot{S}(t)=\Psi(X_1^*(t),t)\dot{X}^*_1(t)-\Psi(X_0^*(t),t)\dot{X}^*_0(t) +\int_{X_0^*(t)}^{X_1^*(t)}\Psi_t(\sigma,t)d\sigma \end{equation} 
\begin{equation}
=\Psi(X_1^*(t),t)\dot{X}^*_1(t)-\Psi(X_0^*(t),t)\dot{X}^*_0(t) + \frac{\Psi_{x^*}(X_1^*(t),t)}{\Psi^2(X_1^*(t),t)}+2\delta X_1^*(t)-\frac{\Psi_{x^*}(X_0^*(t),t)}{\Psi^2(X_0^*(t),t)}-2\delta X_0^*(t) \end{equation} 
and taking into account \eqref{c4}(i) and \eqref{c5} we obtain
\begin{equation}\label{esepunto1}
\dot{S}(t)=\frac{-\delta C(t)}{L(t)}\left(\frac{1}{\Psi(X_1^*(t),t)}-\delta X_1^*(t)\right) \end{equation} which is equivalent to \eqref{b7}.
The condition \eqref{b6} yields immediately from \eqref{c4}(iii) and  \eqref{esepunto2}, and \eqref{b8} is obtained from \eqref{c5} and the fact that $x^*(S(t),t)=X_1^*(t)$.

To prove \eqref{b9} we define $P(t)=\exp\left(\int^t_{0}H(\tau)d\tau\right)$ where $H$ is defined by \eqref{ache}. We have
\begin{equation}
ln(P(t))=\int^t_{0}H(\tau)d\tau
\end{equation}
then \begin{equation}
\frac{P'(t)}{P(t)}= -\frac{T_m^3(t)}{L(t)C^{2}(t)}+\delta\frac{T_m(t)}{L(t)}x_y^{*}(S(t),t)-\delta x_y^{*}(S(t),t)+\delta^{2}{X_1^*}^{2}(t)+\delta x_y^{*}(0,t)-\delta^2 {X_0^*}^{2}(t)
\end{equation}
and taking into account \eqref{b7} and \eqref{b8} its obtain
\begin{equation}
\frac{P'(t)}{P(t)}=-\delta x^{*}(S(t),t)\dot{S}(t)-\delta x_y^{*}(S(t),t)+\delta^{2}{X_1^*}^{2}(t)+\delta x_y^{*}(0,t)-\delta^2 {X_0^*}^{2}(t)
=\frac{d}{dt}\left(-\delta\int_0^{S(t)} x^{*}(\sigma,t)d\sigma\right)
\end{equation}
thus
\[
ln(P(t))-ln(P(0))=-\delta\int_0^{S(t)} x^{*}(\sigma,t)d\sigma
\] or equivalently 
\[
P(t)=exp\left(-\delta\int_0^{S(t)} x^{*}(\sigma,t)d\sigma\right).
\]
Therefore 
\eqref{c4}(ii) is equivalent to 
\[
\frac{1}{\Psi(X_0^*(t),t)}-\delta(X_0^*(t))^2 =\frac{-q P(t)}{\delta C(t)}
\]this is 
\[
x^*_y(0,t)-\delta x^{*2}(0,t)=\frac{-q P(t)}{\delta C(t)}
\] that is to say \eqref{b9}
\end{proof} 

In conclusion we have that the solution to the problem \eqref{C}-\eqref{ache} can be obtained from the solution to the Stefan problem \eqref{b1} $T=T(y,t)$, $S=S(t)$ as established by the following theorem

\begin{theorem}\label{teoremaPsi}
The solution to the problem \eqref{C}-\eqref{ache} is obtained from the solution to the Stefan problem \eqref{b1} $T=T(y,t)$, $S=S(t)$
and its parametric expression is given by:
\begin{equation}
\Psi(x^*,t)=\frac{\delta[C(t)-\int_{S(t)}^{y}T(u,t)du]^{2}}{T_y(y,t)[C(t)-\int_{S(t)}^{y}T(u,t)du]+T^{2}(y,t)} \quad X_0^*(t)< x^* < X_1^*(t) \ , \quad t > 0
\end{equation}
\begin{equation}
x^*=\frac{T(y,t)}{\delta[C(t)-\int_{S(t)}^{y}T(u,t)du]}, \quad 0<y<S(t), \quad t>0. 
\end{equation}
where 
\begin{equation}\label{ce}
C(t)=\int_{0}^{t} [L(\tau)-T_m(\tau)]\dot{S}(\tau)d\tau, \quad t>0,
\end{equation}

\begin{equation}\label{front} X_0^*(t)=\frac{T(0,t)}{\delta[C(t)-\int_{S(t)}^{0}T(u,t)du]}  \quad X_1^*(t) =\frac{T_m(t)}{\delta C(t)}.
\end{equation}
\end{theorem}

In the next section we will solve the Stefan problem \eqref{b1}.

\section{Explicit solution for a canonical one-phase Stefan Problem with latent heat and phase-change temperature variable in time}

In \cite{nsdt11}, Salva and Tarzia in the context of Stefan problems with variable latent heat introduced a novel similarity solutions of the classical heat equation \eqref{b1}$_1$, namely
\begin{equation} \label{d1}
T=2\sqrt{t}\ \eta(\xi) \end{equation}
with $\xi=y/2\sqrt{t}$ and
\begin{equation} \label{d2}
\frac{1}{2}\eta''(\xi)+\xi\eta'(\xi)-\eta(\xi)=0 \end{equation}
with general solution
\begin{equation} \label{d3}
\eta(\xi)=A[\ e^{-\xi^2}+\sqrt{\pi}\xi\ \mathrm{erf}\ \xi\ ]+B\xi \end{equation}
where $A,B$ are arbitrary constants. Thus,
\begin{equation} \label{d4}
T=2\sqrt{t}\ [\ A(e^{-\xi^2}+\sqrt{\pi}\ \xi\ \mathrm{erf}\ \xi)+B\xi\ ] \end{equation}
where the boundary condition \eqref{b1}$_4$, on $y=0$ in the Stefan problem \eqref{b1} requires that
\begin{equation} \label{d5}
[\ A(-2\xi\ e^{-\xi^2}+\sqrt{\pi}\ (\mathrm{erf}\ \xi+\xi\ \frac{2}{\sqrt{\pi}}\ e^{-\xi^2})+B\ ]|_{\xi=0}=-q \end{equation}
so that $B=-q$. Here, the moving boundary in \eqref{b1} is taken as $S(t)=2\gamma\sqrt{t}$ whence the condition \eqref{b1}$_2$, on reduction, yields
\begin{equation} \label{d6}
A=-\frac{\gamma}{\sqrt{\pi}\ \mathrm{erf}\gamma}\left(\frac{1}{\sqrt{t}}\right)L(t)+\frac{q}{\sqrt{\pi}\ \mathrm{erf\gamma}} \end{equation}
so that $L(t)\sim\sqrt{t}$. The condition \eqref{b1}$_3$ on the moving boundary $y=S(t)$ shows that
\begin{equation} \label{d7}
2\sqrt{t}\left[\ A(e^{-\gamma^2}+\sqrt{\pi}\ \gamma\ \mathrm{erf}\ \gamma)-q\gamma\right]=T_m(t) 
\end{equation}
so that $T_m(t)\sim\sqrt{t}$, with the result that $L(t)$ and $T_m(t)$ are related according to
\begin{equation} \label{d8c}
\frac{T_m(t)/2\sqrt{t}+q\gamma}{e^{-\gamma^2}+\sqrt{\pi}\ \gamma\ \mathrm{erf}\ \gamma}=\frac{-\gamma L(t)/\sqrt{t}+q}{\sqrt{\pi}\ \mathrm{erf}\ \gamma} \ . \end{equation}
If we take \begin{equation}\label{elete}
L(t)=L_0\sqrt{t} , \qquad T_m(t)=T_{m_0}\sqrt{t},\qquad L_0>T_{m_0}
\end{equation}
then
\begin{equation} \label{d6e}
A=\frac{q-L_0\gamma}{\sqrt{\pi}\mathrm{erf}\gamma} \end{equation}and the solution to \eqref{b1} is given by

\begin{equation} \label{d4}
T(y,t)=\frac{q-L_0\gamma}{\sqrt{\pi}\mathrm{erf}\gamma}\left(2\sqrt{t}e^{-\frac{y^{2}}{4t}}+\sqrt{\pi}\ y\ \mathrm{erf}\ \frac{y}{2\sqrt{t}}\right)-qy \end{equation}
where $\gamma$ must be solution of 
\begin{equation} \label{d8}
\frac{T_{m_0}/2+q/\gamma}{e^{-\gamma^2}+\sqrt{\pi}\ \gamma\ \mathrm{erf}\ \gamma}=\frac{-\gamma L_0+q}{\sqrt{\pi}\ \mathrm{erf}\ \gamma} \ . \end{equation}which is equivalent to
\begin{equation} \label{d81}
G(\gamma)=F(\gamma)
\end{equation}
where 
\begin{equation} 
G(x)=q-L_{0}x,\quad F(x)=\left(\frac{T_{m_0}}{2}+L_0x^{2}\right) e^{x^{2}}\sqrt{\pi}\ \mathrm{erf}\x,\quad x>0
\end{equation} satisfy
\[
G(0)=q>0,\quad G(+\infty)=-\infty,\quad G'(x)<0 ,\quad x>0\]
\[ F(0)=0>0,\quad G(+\infty)=+\infty,\quad F'(x)>0 , \quad x>0.\]

 From properties of functions $F$ and $G$ we obtain that there exists a unique $\gamma$, $0<\gamma<\frac{q}{L_0}$ that \eqref{d81} holds.

We are in a position to establish the following result
\begin{theorem}\label{Stefan particular}
There exists a unique solution to the Stefan problem
\begin{equation} \label{particular}
\begin{array}{l} T_t = T_{yy} \ , \quad 0 < y < S(t) \ , \quad t > 0
\\[2mm]-T_y(S(t),t) = L_0\sqrt{t}\,\ \dot{S}(t) \ , \quad t > 0 \\[2mm]
T(S(t),t)= T_{m_0}\sqrt{t} \ , \quad t > 0 \\[2mm]T_y(0,t)=-q \ , \quad q>0 \ , \quad t > 0 \\[2mm]
S(0) = 0 \ . \end{array} \end{equation} which is given by 
\begin{equation}\label{TT}
T(y,t)=\frac{q-L_0\gamma}{\sqrt{\pi}\mathrm{erf}\gamma}\left(2\sqrt{t}e^{-\frac{y^{2}}{4t}}+\sqrt{\pi}\ y\ \mathrm{erf}\ \frac{y}{2\sqrt{t}}\right)-qy, \quad 0<y<s(t), \quad t>0
\end{equation}  and 
\begin{equation}\label{fronts}
S(t)=2\gamma\sqrt{t}
\end{equation}
where $\gamma$ is the unique solution to \eqref{d8}.
\end{theorem}

\begin{corollary}
 The coefficient $\gamma$ which characterizes the free boundary $S(t)=2\gamma\sqrt{t}$ satisfies the physical condition
\begin{equation}
q>L_0\gamma+\sqrt{\pi}\frac{T_{m_0}}{2}\mathrm{erf}\ \gamma. 
\end{equation}

\end{corollary}
\begin{remark}It is recalled that in standard Stefan problems wherein $L$ and $T_m$ are constants, $\gamma$ is determined through the boundary conditions by a transcendental equation.

\medskip
\end{remark}
\begin{remark}At fixed face $y=0$ the temperature is time dependent and is given by
\begin{equation}\label{t0}
T(0,t)=\frac{2(q-L_0\gamma)}{\sqrt{\pi}\mathrm{erf}\gamma}\sqrt{t}
\end{equation}
\end{remark}
From theorems \eqref{teoremaPsi} and \eqref{Stefan particular} we can establish the following existence and uniqueness result:

\begin{theorem}
There exists a unique solution to the problem 
\begin{equation}\label{C1}\frac{\partial\Psi}{\partial t}=\frac{\partial}{\partial x^*}\left(\frac{\Psi_{x^{*}}}{\Psi^2}\right)+2\delta  \ , \quad X_0^*(t)< x^* < X_1^*(t) \ , \quad t > 0
\end{equation} with the conditions
\begin{equation} \label{c41}
\begin{array}{l}  (i)\quad   \frac{1}{\Psi(X_1^*(t),t)}-\delta{X_1^{*}}^{2}(t) = -\frac{L_0}{\delta \gamma(L_0-T_{m_0})\sqrt{t}}\left(\Psi(X_1^*(t),t) \dot{X_1^*}(t)+\frac{\Psi_{x^*}(X_1^*(t),t)}{\Psi^2(X_1^*(t),t)}+2\delta X_1^*(t)\right)\ , \quad t > 0 \\[3mm]
(ii)  \quad\frac{1}{\Psi(X_0^*(t),t)}-\delta{X_0^*}^{2}(t) =\frac{-q \exp\left(\int^{t}_{0}H(\tau)d\tau\right)}{\delta \gamma(L_0-T_{m_0})t} \ , \quad q>0 \ , \quad t > 0 \\[3mm]
(iii) \quad \frac{\gamma}{\sqrt{t}}=\Psi(X_1^*(t),t) \dot{X_1^*}(t)+\frac{\Psi_{x^*}(X_1^*(t),t)}{\Psi^2(X_1^*(t),t)}+2\delta X_1^*(t) \ , \ , \quad t > 0  \end{array} \end{equation}
where the moving boundaries $X_0^*$ and $X_1^*$ must satisfy
\begin{equation} \label{c51}
X_0^*(t)=-\int^t_{0} \left(\frac{\Psi_{x^*}(X_0^*(\tau),\tau)}{\Psi^3(X_0^*(\tau),\tau)}+2\delta \frac{X_0^*(\tau)}{\Psi(X_0^*(\tau),\tau)}\right)d\tau \quad,\quad X_1^*(t)= \frac{T_{m_0}}{\delta \gamma(L_0-T_{m_0})\sqrt{t} }\ , \quad t > 0\end{equation}
and
\begin{equation}\label{ache1}
H(t)=-\frac{T_{m_0}^3}{L_0\gamma^{2}(L_0-T_{m_0})^{2}}+\delta\frac{T_{m_0}}{L_0}\frac{1}{\Psi(X_1^*(t),t)}-\frac{\delta}{\Psi(X_1^*(t),t)}+\frac{T_{m_0}^{2}}{\gamma^{2}(L_0-T_{m_0})^{2}t}+\frac{\delta}{\Psi(X_0^*(t),t)}-\delta^2 {X_0^*}^{2}(t)
\end{equation}
which is given by

\begin{equation}\label{sol}
\Psi(x^*,t)=\frac{\delta\Theta^{2}(y,t)}{T_y(y,t)\Theta(y,t)+T^{2}(y,t)}, \quad X_0^*(t)< x^* < X_1^*(t) \ , \quad t > 0
\end{equation}
\begin{equation}\label{sol1}
x^*=\frac{\frac{q-L_0\gamma}{\sqrt{\pi}\mathrm{erf}\gamma}\left(2\sqrt{t}e^{-\frac{y^{2}}{4t}}+\sqrt{\pi}\ y\ \mathrm{erf}\ \frac{y}{2\sqrt{t}}\right)-qy}{\delta\Theta(y,t)}, \quad 0<y<S(t), \quad t>0. 
\end{equation}
where $T$ is given by \eqref{TT}

\begin{equation}\label{ce}
\Theta(y,t)=\gamma(L_0-T_{m_0})t+2q\left(\frac{y^{2}}{4t}-\gamma^{2}\right)t
\end{equation}
\[
-2\frac{q-L_0\gamma}{\sqrt{\pi}\mathrm{erf}\gamma}\left[\frac{\sqrt{\pi}}{2}\left(\mathrm{erf}\frac{y}{2\sqrt{t}}-\mathrm{erf}\gamma\right)+\sqrt{\pi}\left(\frac{y^{2}}{4t}\mathrm{erf}\frac{y}{2\sqrt{t}}-\gamma^{2}\mathrm{erf}\gamma\right)+\frac{y}{2\sqrt{t}}e^{-\frac{y^{2}}{4t}}-\gamma e^{-\gamma^{2}}\right]t\]
and the free boundaries are
\begin{equation}\label{front0} 
X_0^*(t)=\frac{C_0}{\delta\sqrt{t}},
  \qquad X_1^*(t) =\frac{C_1}{\delta\sqrt{t}}
\end{equation} with
\begin{equation}
C_0=\frac{2\frac{q-L_0\gamma}{\sqrt{\pi}\mathrm{erf}\gamma}}{\gamma(L_0-T_{m_0})-2q\gamma^{2}+2\frac{q-L_0\gamma}{\sqrt{\pi}\mathrm{erf}\gamma}\left(\frac{\sqrt{\pi}}{2}\ \mathrm{erf}\ \gamma+\sqrt{\pi}\gamma^{2}\ \mathrm{erf}\ \gamma+\gamma e^{-\gamma^{2}}\right) },\quad C_1=\frac{T_{m_0}}{\gamma(L_0-T_{m_0})}.
\end{equation} and

\end{theorem}
\begin{proof}
Here, with $L(t)$ and $T_{m}(t)$ given by \eqref{elete}, on insertion in \eqref{b6}, integration shows that $C(t)$ is linear in $t$ is this
\begin{equation}\label{ce1}
C(t)=\gamma(L_0-T_{m_0})t.
\end{equation}
Moreover, by considering \eqref{TT}, \eqref{fronts} and 
$$ \int x\mathrm{erf}(x) dx=\frac{x^{2}}{2}\mathrm{erf}(x)+\frac{1}{2\sqrt{t}}\mathrm{exp}(-x^{2})x-\frac{\mathrm{erf}(x)}{4}$$  we obtain
\begin{equation}\label{t1}
\int_{S(t)}^{y}T(\sigma,t)d\sigma=-2q\left(\frac{y^{2}}{4t}-\gamma^{2}\right)t+
\end{equation}
\[
+2A\left[\frac{\sqrt{\pi}}{2}\left(\mathrm{erf}\frac{y}{2\sqrt{t}}-\mathrm{erf}\gamma\right)+\sqrt{\pi}\left(\frac{y^{2}}{4t}\mathrm{erf}\frac{y}{2\sqrt{t}}-\gamma^{2}\mathrm{erf}\gamma\right)+\frac{y}{2\sqrt{t}}e^{-\frac{y^{2}}{4t}}-\gamma e^{-\gamma^{2}}\right]t.\]

Taking into account \eqref{t0}, \eqref{ce1} and \eqref{t1} we obtain \eqref{sol} and \eqref{sol1} where function 

\begin{equation}\label{ce}
\Theta(y,t):=C(t)-\int_{S(t)}^{y}T(\sigma,t)d\sigma=\gamma(L_0-T_{m_0})t+2q\left(\frac{y^{2}}{4t}-\gamma^{2}\right)t
\end{equation}
\[
-2\frac{q-L_0\gamma}{\sqrt{\pi}\mathrm{erf}\gamma}\left[\frac{\sqrt{\pi}}{2}\left(\mathrm{erf}\frac{y}{2\sqrt{t}}-\mathrm{erf}\gamma\right)+\sqrt{\pi}\left(\frac{y^{2}}{4t}\mathrm{erf}\frac{y}{2\sqrt{t}}-\gamma^{2}\mathrm{erf}\gamma\right)+\frac{y}{2\sqrt{t}}e^{-\frac{y^{2}}{4t}}-\gamma e^{-\gamma^{2}}\right]t.\]

Rewritten \eqref{front}, we obtain \eqref{front0} 

\end{proof}
\def\bitem{\vspace{-0.2cm}\bibitem}
\def\fit#1{\textit{\frenchspacing#1}}

\section{Conclusions}
The equivalence of a new Stefan-problem with latent heat and phase-change temperature with a nonlinear evolution equation with a source term and two free boundaries is obtained. For a particular case a unique explicit solution for both free boundary problems are also obtained.

\section*{Acknowledgement}

The project has received funding from the European Union’s Horizon 2020
research and innovation programme under the Marie Skłodowska-Curie
Grant Agreement No. 823731 CONMECH and Project O06-INV00025 Universidad Austral.

\end{document}